\documentclass[12pt]{article}

\usepackage[a4paper]{geometry}
\usepackage{amsthm}
\newcommand{\supplementary}{Appendix~\ref{app:pack}}
\newcommand{\appsection}[1]{\subsection{#1}}
\newcommand{\appSection}{Subsection}
\newcommand{\appref}[3]{#2~\ref{#1}}
\newcommand{\appeqref}[2]{\eqref{#1}}

\newtheorem{definition}{Definition}
\newtheorem{theorem}[definition]{Theorem}

\newtheorem{lemma}[definition]{Lemma}

\theoremstyle{definition}

\newcommand{\dateline}[3]{}
\newcommand{\MSC}[1]{\newcommand{\msc}{#1}}
\newcommand{\Copyright}[1]{\newcommand{\license}{#1}}
\newcommand{\datemsc}{\date{\today\\ \small Mathematics Subject Classification: \msc \\ \copyright\thinspace\ignorespaces\license}}
\newcommand{\doi}[1]{\href{http://dx.doi.org/#1}{\texttt{doi:#1}}}
\newcommand{\arxiv}[1]{\href{http://arxiv.org/abs/#1}{\texttt{arXiv:#1}}}

\usepackage[english]{babel}
\usepackage[utf8x]{inputenc}
\usepackage{amsmath, amssymb, amsfonts, tikz, mathtools}
\usepackage[footnotesize]{caption}
\usepackage[titletoc,title]{appendix}
\usepackage{array}
\usepackage{booktabs}
\usepackage[numbers]{natbib}
\usepackage{etoolbox}
\usepackage[colorlinks=true,citecolor=black,linkcolor=black,urlcolor=blue]{hyperref}
\usepackage{color}
\usepackage{listings}
\usepackage{setspace}

\usetikzlibrary{shapes}

\definecolor{deepblue}{rgb}{0,0,0.5}
\definecolor{deepred}{rgb}{0.6,0,0}
\definecolor{deepgreen}{rgb}{0,0.5,0}

\DeclareFixedFont{\ttb}{T1}{txtt}{bx}{n}{10} 
\DeclareFixedFont{\ttm}{T1}{txtt}{m}{n}{10}  

\newcommand{\pythonstyle}{\lstset{
language=Python,
basicstyle=\ttm,
otherkeywords={self},             
keywordstyle=\ttb\color{deepblue},
emph={sage,__init__,True,False,None},  
emphstyle=\ttb\color{deepred},    
stringstyle=\color{deepgreen},
frame=tb,                         
showstringspaces=false
}}

\lstnewenvironment{python}[1][]
{
\pythonstyle
\lstset{#1}
}
{}

\newcommand\pythoninline[1]{{\pythonstyle\lstinline!#1!}}

\newcolumntype{x}[1]{%
>{\centering\arraybackslash}m{#1}}%

\newcommand{\noop}[1]{}

\newcommand{\G}{\ensuremath{\Gamma}}
\newcommand{\M}{\ensuremath{\mathcal{M}}}

\newcommand{\sV}[2]{{
	\setlength{\arraycolsep}{2pt}
	\renewcommand{\arraystretch}{0.8}
	\left[\begin{array}{ccc} #1 \\ #2 \end{array}\right]
}}

\begin{document}

\dateline{Mar 30, 2018}{Sep 27, 2018}{Oct 19, 2018}
\MSC{05E30}
\Copyright{Janoš Vidali. Released under the CC BY license (International 4.0).}


\title{Using symbolic computation\\
to prove nonexistence of distance-regular graphs}
\author{Janoš Vidali%
\thanks{This work is supported in part
by the Slovenian Research Agency (research program P1-0285).}
\\
\small Faculty of Mathematics and Physics\\[-0.8ex]
\small University of Ljubljana, 1000 Ljubljana, Slovenia\\
\small\tt janos.vidali@fmf.uni-lj.si}
\datemsc

\maketitle

{
\begin{abstract}
\setlength{\abovedisplayskip}{0pt}
\setlength{\belowdisplayskip}{3pt}
A package for the Sage computer algebra system is developed
for checking feasibility of a given intersection array
for a distance-regular graph.
We use this tool to show that
there is no distance-regular graph
with intersection array
\begin{multline*}
\{(2r+1)(4r+1)(4t-1), 8r(4rt-r+2t), (r+t)(4r+1); \\[-3pt]
1, (r+t)(4r+1), 4r(2r+1)(4t-1)\} \quad (r, t \ge 1),
\end{multline*}
$\{135,\! 128,\! 16; 1,\! 16,\! 120\}$, $\{234,\! 165,\! 12; 1,\! 30,\! 198\}$
or $\{55,\! 54,\! 50,\! 35,\! 10; 1,\! 5,\! 20,\! 45,\! 55\}$.
In all cases, the proofs rely on equality in the Krein condition,
from which triple intersection numbers are determined.
Further combinatorial arguments are then used to derive nonexistence.

\end{abstract}
}

\section{Introduction}\label{sec:int}

Distance-regular graphs were introduced around 1970 by N.~Biggs~\cite{b71}.
As they are intimately linked to many other combinatorial objects,
such as finite simple groups, finite geometries, and codes,
a natural goal is trying to classify them.

Many distance-regular graphs are known,
however constructing new ones has proved to be a difficult task.
A number of feasibility conditions
for distance-regular graphs have been found,
which allows us to compile a list of feasible intersection arrays
for small distance-regular graphs
(or related structures, such as $Q$-polynomial association schemes),
see Brouwer et al.~\cite{b11,bcn89,bcn94} and Williford~\cite{w17a}.
However, feasibility is no guarantee for existence,
so proofs of nonexistence of distance-regular graphs
with feasible intersection arrays
are also a contribution to the classification.
In certain cases,
single intersection arrays have been ruled out~\cite{k92,l93},
while other proofs may show nonexistence for a whole infinite family
of feasible intersection arrays~\cite{cj08,jv12,u12}.
In this paper we give proofs of nonexistence for distance-regular graphs
belonging to a two-parameter infinite family,
as well as for graphs with intersection arrays
\begin{align*}
\{135, 128, 16;&\ 1, 16, 120\}, \\
\{234, 165, 12;&\ 1, 30, 198\}, \\
\{55, 54, 50, 35, 10;&\ 1, 5, 20, 45, 55\}.
\end{align*}

We develop a package called {\tt sage-drg}~\cite{v18}
for the Sage computer algebra system~\cite{sage17}.
Sage is free open-source software
written in the Python programming language~\cite{p17},
with many functionalities deriving from other free open-source software,
such as Maxima~\cite{maxima17}, which Sage uses for symbolic computation.
The {\tt sage-drg} package is thus also free open-source software
available under the MIT license,
written in the Python programming language,
making use of the Sage library.
The package can be used to check for feasibility
of a given intersection array against known feasibility conditions,
see Van Dam, Koolen and Tanaka for an up-to-date survey~\cite{vdkt16}.
Furthermore, using equality in the Krein condition
(see Theorem~\ref{thm:krein0}),
restrictions on triple intersection numbers can be derived.
In this paper, we use them to derive some nonexistence results.
The {\tt sage-drg} package also includes Jupyter notebooks
demonstrating its use to obtain these results,
as well as the notebook
\href{https://github.com/jaanos/sage-drg/blob/master/jupyter/Demo.ipynb}{\tt jupyter/Demo.ipynb}
giving some general examples of use of the package.
A more detailed description of the {\tt sage-drg} package
is given in \supplementary.

The results from Sections~\ref{sec:2param},~\ref{sec:135} and~\ref{sec:5554}
appeared in the author's PhD thesis~\cite{v13},
where computation was done using a Mathematica~\cite{w10} notebook
originally developed by M.~Ur\-lep.
Thus, the {\tt sage-drg} package can be seen as a move
from closed-source proprietary software to free open-source software,
which allows one to check all code for correctness,
thus making the results verifiable.

\section{Preliminaries}\label{sec:pre}

In this section we review some basic definitions and concepts.
See Brouwer, Cohen and Neumaier~\cite{bcn89} for further details.

Let $\G$ be a connected graph with diameter $d$ and $n$ vertices,
and let $\partial(u, v)$ denote
the distance between the vertices $u$ and $v$ of $\G$.
The graph $\G$ is {\em distance-regular}
if there exist constants $p^h_{ij}$ ($0 \le h, i, j \le d$),
called the {\em intersection numbers},
such that for any pair of vertices $u, v$ at distance $h$,
there are precisely $p^h_{ij}$ vertices
at distances $i$ and $j$ from $u$ and $v$, respectively.
In fact, all intersection numbers can be computed
given only the intersection numbers
$b_i = p^i_{1,i+1}$ and $c_{i+1} = p^{i+1}_{1,i}$ ($0 \le i \le d-1$)~%
\cite[\S4.1A]{bcn89}.
These intersection numbers are usually gathered in the
{\em intersection array} $\{b_0, b_1, \dots, b_{d-1}; c_1, c_2, \dots, c_d\}$.
We also define the {\em valency} $k = b_0$
and $a_i = k - b_i - c_i$ ($0 \le i \le d$), where $b_d = c_0 = 0$.
A connected noncomplete {\em strongly regular} graph
with parameters $(v, k, \lambda, \mu)$
is a distance-regular graph of diameter $2$ with $v$ vertices,
valency $k$ and intersection numbers $a_1 = \lambda$, $c_2 = \mu$.

Let $A_i$ ($0 \le i \le d$) be a binary square matrix
indexed with the vertices of a graph $\G$ of diameter $d$,
with entry corresponding to vertices $u$ and $v$ equal to $1$
precisely when $\partial(u, v) = i$.
The matrix $A = A_1$ is the {\em adjacency matrix} of $\G$.
The graph $\G$ is called {\em primitive}
if all $A_i$ ($1 \le i \le d$) are adjacency matrices of connected graphs.
A distance-regular graph of valency $k \ge 3$ that is not primitive
is bipartite or antipodal (or both)~\cite[Thm.~4.2.1]{bcn89}.
The spectrum of $\G$ is defined to be the spectrum of $A$
(i.e., eigenvalues with multiplicities)
and can be computed directly from the intersection array of $\G$~%
\cite[\S4.1B]{bcn89}.

Suppose that $\G$ is distance-regular.
Let $\M$ be the {\em Bose-Mesner} algebra,
i.e., the algebra generated by $A$.
The matrices $\{A_i\}_{i=0}^d$ form a basis of $\M$,
which also has a second basis $\{E_i\}_{i=0}^d$
consisting of projectors to the eigenspaces of $A$~\cite[\S2.2]{bcn89}.
Note that the indexing in this second basis
depends on the ordering of eigenvalues.
The descending ordering of eigenvalues is known as the {\em natural ordering}.
We define the {\em eigenmatrix} $P$ and {\em dual eigenmatrix} $Q$
as $(d+1) \times (d+1)$ matrices such that
$A_j = \sum_{i=0}^d P_{ij} E_i$ and $E_j = n^{-1} \sum_{i=0}^d Q_{ij} A_i$.
The graph $\G$ is called {\em formally self-dual}~\cite[p.~49]{bcn89}
if $P = Q$ holds for some ordering of eigenvalues.
The graph $\G$ is called {\em $Q$-polynomial}~\cite[\S2.7]{bcn89}
with respect to some ordering of eigenvalues
if there exist real numbers $z_0, \dots, z_d$
and polynomials $q_j$ of degree $j$
such that $Q_{ij} = q_j(z_i)$ ($0 \le i, j \le d$).
Finally, we define the {\em Krein parameters} $q^h_{ij}$~\cite[\S2.3]{bcn89}
as such numbers that $E_i \circ E_j = n^{-1} \sum_{h=0}^d q^h_{ij} E_h$,
where $\circ$ represents entrywise multiplication of matrices.
A formally self-dual distance-regular graph is also $Q$-polynomial
with respect to the corresponding ordering of eigenvalues
and has $p^h_{ij} = q^h_{ij}$ ($0 \le i, j, h \le d$).
In this paper, we will use the natural ordering for indexing,
noting when a graph is $Q$-polynomial or formally self-dual
for some other ordering.

For vertices $u$, $v$, $w$ of the distance-regular graph $\G$
and integers $i$, $j$, $h$ ($0 \le i,j,h \le d$)
we denote by $\sV{u & v & w}{i & j & h}$
(or simply $[i\ j\ h]$ when it is clear
which triple $(u,v,w)$ we have in mind)
the number of vertices of $\G$
that are at distances $i$, $j$, $h$ from $u$, $v$, $w$, respectively.
We call these numbers {\em triple intersection numbers}.
They have first been studied
in the case of strongly regular graphs~\cite{cgs78},
and later also for distance-regular graphs,
see for example~\cite{cj08,jkt00,jv12,jv17,u12}.
Unlike the intersection numbers,
these numbers may depend on the particular choice of vertices $u, v, w$
and not only on their pairwise distances.
We may however write down a system of $3d^2$ linear Diophantine equations
with $d^3$ triple intersection numbers as variables,
thus relating them to the intersection numbers, cf.~\cite{jv12}:
{\small
\begin{equation}
\sum_{\ell=1}^d [\ell\ j\ h] = p^U_{jh} - [0\ j\ h], \qquad
\sum_{\ell=1}^d [i\ \ell\ h] = p^V_{ih} - [i\ 0\ h], \qquad
\sum_{\ell=1}^d [i\ j\ \ell] = p^W_{ij} - [i\ j\ 0],
\label{eqn:triple}
\end{equation}
}
where $U = \partial(v, w)$, $V = \partial(u, w)$, $W = \partial(u, v)$, and
$$
[0\ j\ h] = \delta_{jW} \delta_{hV}, \qquad
[i\ 0\ h] = \delta_{iW} \delta_{hU}, \qquad
[i\ j\ 0] = \delta_{iV} \delta_{jU}.
$$
Furthermore, we can use the triangle inequality
to conclude that certain triple intersection numbers must be zero.
Moreover, the following theorem sometimes gives additional equations.

\begin{theorem}\label{thm:krein0}{\rm (\cite[Theorem~3]{cj08},
                                   cf.~\cite[Theorem~2.3.2]{bcn89})}
Let $\G$ be a distance-regular graph with diameter~$d$,
dual eigenmatrix $Q$ and Krein parameters $q_{ij}^h$ $(0\le i,j,h \le d)$.
Then,
\[
\pushQED{\qed}
q^h_{ij} = 0 \quad \Longleftrightarrow \quad
\sum_{r,s,t=0}^d Q_{ri}Q_{sj}Q_{th}\sV{u & v & w}{r & s & t} = 0
\quad \mbox{for all\ \ } u, v, w \in V\G.
\qedhere
\popQED
\]
\end{theorem}

Together with integrality and nonnegativity of triple intersection numbers,
we can use all of the above to either derive
that the system of equations has no solution,
or arrive at a small number of solutions,
which gives us new information on the structure of the graph
and may lead to proving its nonexistence.

\section[A two-parameter family of primitive graphs of diameter $3$]%
{A two-parameter family of primitive graphs of diameter $\boldsymbol{3}$}
\label{sec:2param}

In~\cite{jv12},
graphs meeting necessary conditions for the existence of extremal codes
were studied.
One of the families of primitive graphs of diameter $3$
for which these conditions were met was
\begin{equation}
\label{eqn:famcode}
\{a(p+1), (a+1)p, c; 1, c, ap\} ,
\end{equation}
where $a = a_3$, $c = c_2$ and $p = p^3_{33}$.
Graphs belonging to this family are $Q$-polynomial
with respect to the natural ordering of eigenvalues
precisely when the Krein parameter $q^3_{11}$ is zero,
which is equivalent to
\begin{equation}
\label{eqn:famqpoly}
c = {1 \over 4} \left((p+1)^2 + {2a(p+1) \over p+2} \right) .
\end{equation}
Hence, $p+2$ must divide $2a$ for $c$ to be integral.
If $p = 2r-1$, then $a = t(2r+1)$ and $c = r(r+t)$
for some positive integers $r, t$,
which gives us the two-parameter family
$$
\{2rt(2r+1), (2r-1)(2rt+t+1), r(r+t); 1, r(r+t), t(4r^2-1)\} .
$$
In~\cite{jv12},
nonexistence was shown for a feasible subfamily with $r = t \ge 2$.
If, on the other hand, $p$ is even,
integrality of the multiplicity of the second largest eigenvalue
implies that we must have $p = 4r$,
$a = (2r+1)(4t-1)$ and $c = (r+t)(4r+1)$ for some positive integers $r, t$,
giving the family
\begin{equation}
\label{eqn:fameven}
\begin{multlined}
\{(2r+1)(4r+1)(4t-1), 8r(4rt-r+2t), (r+t)(4r+1); \\
1, (r+t)(4r+1), 4r(2r+1)(4t-1)\} .
\end{multlined}
\end{equation}
We find two one-parameter infinite subfamilies
of feasible intersection arrays by setting $t = 4r^2$ or $t = 4r^2 + 2r$:
\begin{gather*}
\begin{multlined}
\{(2r+1)(4r+1)(16r^2-1), 8r^2(16r^2+8r-1), r(4r+1)^2; \\
1, r(4r+1)^2, 4r(2r+1)(16r^2-1)\},
\end{multlined} \\
\begin{multlined}
\{(2r+1)(4r+1)(16r^2+8r-1), 8r^2(4r+1)(4r+3), r(4r+1)(4r+3); \\
1, r(4r+1)(4r+3), 4r(2r+1)(16r^2+8r-1)\}.
\end{multlined}
\end{gather*}
There are also other feasible cases
-- for instance, when $r = 2$,
we have, besides the cases from the two subfamilies above,
feasible examples when $t \in \{4, 7, 196\}$.
The case with $r = 1$ and $t = 4$ belonging to the first subfamily above
is also listed in the list of feasible parameter sets
for $3$-class $Q$-polynomial association schemes
by J.~S.~Williford~\cite{w17a}.

We now prove that a graph $\Delta$
with intersection array \eqref{eqn:fameven} does not exist.
The proof parallels that of~\cite[Lems.~1, 3]{jv12}
-- in fact,
a significant part of the proof
may be extended to the entire family \eqref{eqn:famqpoly},
as it has been done in~\cite{v13}.
The computation needed to obtain the results in this section
is illustrated in the
\href{https://github.com/jaanos/sage-drg/blob/master/jupyter/DRG-d3-2param.ipynb}{\tt jupyter/DRG-d3-2param.ipynb}
notebook included in the {\tt sage-drg} package~\cite{v18}.

\begin{lemma} \label{lem:2param:333}
Let $\Delta$ be a distance-regular graph
with intersection array \eqref{eqn:fameven},
and $u', v, w$ be vertices of $\Delta$
with $\partial(u', v) = 1$, $\partial(u', w) = 2$ and $\partial(v, w) = 3$.
Then $\sV{u' & v & w}{3 & 3 & 3} = 1$.
\end{lemma}

\begin{proof}
Let $u$ be a vertex of $\Delta$ at distance $3$ from both $v$ and $w$
(such a vertex exists since $p^3_{33} = 4r > 0$).
We consider the triple intersection numbers $[i\ j\ h]$
that correspond to $(u, v, w)$.
As $q^3_{11} = q^1_{13} = q^1_{31} = 0$,
Theorem~\ref{thm:krein0} gives three additional equations
to the system \eqref{eqn:triple},
allowing us to express its solution
in terms of a single parameter $\alpha = [3\ 3\ 3]$.
Let us express the counts of vertices
at distance $1$ or $2$ from one of $u, v, w$
and at distance $3$ from the other two vertices:
\begin{align*}
[3\ 3\ 1] = [3\ 1\ 3] = [1\ 3\ 3] &= {(\alpha - 4r + 1)(4r+1) \over 4r-1} , \\
[3\ 3\ 2] = [3\ 2\ 3] = [2\ 3\ 3] &= {8r(4r - 1 - \alpha) \over 4r-1} .
\end{align*}
For the values above to be nonnegative, we must have $\alpha = 4r-1$,
which means that they are all zero.
As the choice of $u, v, w$ was arbitrary,
this implies that any pair of vertices at distance $3$
induces a set of $4r+2$ vertices pairwise at distance $3$
-- in the terminology of~\cite{jv12},
this is a maximal $1$-code in $\Delta$.
Since we have $a_3 p^3_{33} = 4r(2r+1)(4t-1) = c_3$,
it follows by~\cite[Prop.~2]{jv12}
that $\sV{u' & v & w}{3 & 3 & 3} = 1$ holds.
\end{proof}

\begin{theorem} \label{thm:2param:nonex}
A distance-regular graph $\Delta$ with intersection array \eqref{eqn:fameven}
does not exist.
\end{theorem}

\begin{proof}
Let $u', v, w$ be vertices of $\Delta$ with
$\partial(u', v) = 1$, $\partial(u', w) = 2$ and $\partial(v, w) = 3$
(such vertices exist, since we have $p^2_{13} = b_2 = (r+t)(4r+1) > 0$).
We consider the triple intersection numbers $[i\ j\ h]$
that correspond to $(u', v, w)$.
By Lemma~\ref{lem:2param:333}, we have $[3\ 3\ 3] = 1$.
Using $q^3_{11} = 0$,
Theorem~\ref{thm:krein0} gives an additional equation
which allows us to obtain a unique solution to the system \eqref{eqn:triple}.
However, we obtain $[1\ 1\ 3] = 2t - 1/2$,
which is nonintegral for all integers $t$.
Therefore, the graph $\Delta$ does not exist.
\end{proof}

\section[A primitive graph with diameter $3$ and $1360$ vertices]%
{A primitive graph with diameter $\boldsymbol{3}$
and $\boldsymbol{1360}$ vertices}\label{sec:135}

Let $\Lambda$ be a distance-regular graph with intersection array
\begin{equation}
\{135, 128, 16; 1, 16, 120\}. \label{eqn:135:ia}
\end{equation}
This intersection array can be obtained from \eqref{eqn:famcode}
by setting $a = 15$, $c = 16$ and $p = 8$.
The graph $\Lambda$ has diameter $3$ and $1360$ vertices.
It is not $Q$-polynomial, however its Krein parameter $q^3_{33}$ is zero.
We show that such a graph does not exist.
The computation needed to prove Theorem~\ref{thm:135:nonex}
is illustrated in the
\href{https://github.com/jaanos/sage-drg/blob/master/jupyter/DRG-135-128-16-1-16-120.ipynb}{\tt jupyter/DRG-135-128-16-1-16-120.ipynb}
notebook included in the {\tt sage-drg} package~\cite{v18}.

\begin{theorem} \label{thm:135:nonex}
A distance-regular graph $\Lambda$ with intersection array \eqref{eqn:135:ia}
does not exist.
\end{theorem}

\begin{proof}
Let $u, v, w$ be three pairwise adjacent vertices of $\Lambda$
(such vertices exist, since we have $p^1_{11} = 6 > 0$).
We consider triple intersection numbers $[i\ j\ h]$
that correspond to $(u, v, w)$.
As $q^3_{33} = 0$,
Theorem~\ref{thm:krein0} gives an additional equation
to the system \eqref{eqn:triple},
allowing us to express its solution
in terms of a single parameter $\alpha = [1\ 1\ 1]$.
In particular, we obtain
$$
[3\ 3\ 3] = {71 - 27\alpha \over 8} .
$$
Clearly, $\alpha$ must be a nonnegative integer.
For $[3\ 3\ 3]$ to be nonnegative, we must have $\alpha \in \{0, 1, 2\}$.
However, $[3\ 3\ 3]$ is still nonintegral in these cases,
showing that the graph $\Lambda$ does not exist.
\end{proof}

\section[A primitive graph with diameter $3$ and $1600$ vertices]%
{A primitive graph with diameter $\boldsymbol{3}$
and $\boldsymbol{1600}$ vertices}\label{sec:234}

Let $\Xi$ be a distance-regular graph with intersection array
\begin{equation}
\{234, 165, 12; 1, 30, 198\}. \label{eqn:234:ia}
\end{equation}
The graph $\Xi$ has diameter $3$ and $1600$ vertices.
The intersection array \eqref{eqn:234:ia} has been found
as an example of a feasible parameter set for a distance-regular graph
which is formally self-dual for an ordering of eigenvalues
distinct from the natural ordering
-- in fact, $\Xi$ is $Q$-polynomial for the ordering $0, 2, 3, 1$,
so its Krein parameters $q^1_{22}$, $q^2_{12}$ and $q^2_{21}$ are zero.
The intersection array \eqref{eqn:234:ia}
is also listed in the list of feasible parameter sets
for $3$-class $Q$-polynomial association schemes
by J.~S.~Williford~\cite{w17a}.
We show that such a graph does not exist.
The computation needed to prove Theorem~\ref{thm:234:nonex}
is illustrated in the
\href{https://github.com/jaanos/sage-drg/blob/master/jupyter/DRG-234-165-12-1-30-198.ipynb}{\tt jupyter/DRG-234-165-12-1-30-198.ipynb}
notebook included in the {\tt sage-drg} package~\cite{v18}.

\begin{theorem} \label{thm:234:nonex}
A distance-regular graph $\Xi$ with intersection array \eqref{eqn:234:ia}
does not exist.
\end{theorem}

\begin{proof}
Let $u, v, w$ be three vertices of $\Xi$ that are pairwise at distance $3$
(such vertices exist, since we have $p^3_{33} = 8 > 0$).
We consider triple intersection numbers $[i\ j\ h]$
that correspond to $(u, v, w)$.
As $q^1_{22} = q^2_{12} = q^2_{21} = 0$,
Theorem~\ref{thm:krein0} gives three additional equations
to the system \eqref{eqn:triple},
allowing us to express its solution
in terms of a single parameter $\alpha = [3\ 3\ 3]$.
In particular, we obtain
$$
[3\ 3\ 2] = [3\ 2\ 3] = [2\ 3\ 3] = -17 - 4\alpha .
$$
Clearly, $\alpha$ must be nonnegative,
but then we have $[3\ 3\ 2] = [3\ 2\ 3] = [2\ 3\ 3] < 0$, a contradiction.
We conclude that the graph $\Xi$ does not exist.
\end{proof}

\section[A bipartite graph with diameter $5$]%
{A bipartite graph with diameter $\boldsymbol{5}$}\label{sec:5554}

Let $\Sigma$ be a distance-regular graph with intersection array
\begin{equation}
\{55, 54, 50, 35, 10; 1, 5, 20, 45, 55\}. \label{eqn:5554:ia}
\end{equation}
This intersection array appears
in the list of feasible intersection arrays
for bipartite non-antipodal distance-regular graphs of diameter $5$
by Brouwer et.~al.~\cite[p.~418]{bcn89} as an open case.
The existence of such a graph would give a counterexample
to a conjecture by MacLean and Terwilliger~\cite{mlt06}, cf.~Lang~\cite{l14}.
The computation needed to obtain the results in this section
is illustrated in the
\href{https://github.com/jaanos/sage-drg/blob/master/jupyter/DRG-55-54-50-35-10-bipartite.ipynb}{\tt jupyter/DRG-55-54-50-35-10-bipartite.ipynb}
notebook included in the {\tt sage-drg} package~\cite{v18}.

The graph $\Sigma$ has diameter $5$ and $3500$ vertices.
The partition of $\Sigma$ corresponding to two vertices at distance $2$
is shown in Figure~\ref{fig:dp22}.
The graph is $Q$-polynomial for the natural ordering of eigenvalues,
see for example~\cite[p.~418]{bcn89}.
Moreover, as the graph is bipartite,
it is also $Q$-antipodal~\cite[Thm.~8.2.1]{bcn89}.
Many Krein parameters are zero,
in particular $q^3_{11}$ and $q^4_{11}$ due to the triangle inequality.
We use this fact in the proof of the following statement.

\begin{figure}[t]
	\centering
	\beginpgfgraphicnamed{dp22-d5b}
	\begin{tikzpicture}[style=thick,scale=0.6]
		\scriptsize
        \renewcommand{\arraystretch}{1.2}
        \setlength{\arraycolsep}{2pt}
        \tikzstyle{every node}=[]
        \tikzstyle{vertex}=[draw, circle, inner sep=0pt, minimum size=1mm]
        \tikzstyle{baloon}=[draw, circle, inner sep=0pt, minimum size=7mm]
        \tikzstyle{edge}=[auto]

        \node   (x) at (-6, 2.5) [vertex,label=135:$u$] {};
        \node   (y) at (-6,-2.5) [vertex,label=225:$v$] {};
        \node (D11) at (-6, 0) [baloon,label=below:] {5};
        \node (D13) at (-3, 2.5) [baloon,label=above:] {50};
        \node (D31) at (-3,-2.5) [baloon,label=below:] {50};
        \node (D22) at (-3, 0) [baloon,label=below:] {243};
        \node (D24) at ( 0, 2.5) [baloon,label=below:] {350};
        \node (D42) at ( 0,-2.5) [baloon,label=below:] {350};
        \node (D33) at ( 0, 0) [baloon,label=below:] {1260};
        \node (D35) at ( 3, 2.5) [baloon,label=below:] {175};
        \node (D53) at ( 3,-2.5) [baloon,label=below:] {175};
        \node (D44) at ( 3, 0) [baloon,label=below:] {805};
        \node (D55) at ( 6, 0) [baloon,label=below:] {35};

        \draw (x) -- (D11)
            node[edge, very near start, swap] {$5$}
            node[edge, very near end, swap] {$1$};
        \draw (x) -- (D13)
            node[edge, very near start] {$50$}
            node[edge, very near end] {$1$};
        \draw (y) -- (D11)
            node[edge, very near start] {$5$}
            node[edge, very near end] {$1$};
        \draw (y) -- (D31)
            node[edge, very near start, swap] {$50$}
            node[edge, very near end, swap] {$1$};
        \draw (D11) -- (D22)
            node[edge, very near start] {$53$};
        \draw (D13) -- (D22)
            node[edge, very near start, swap] {$19$};
        \draw (D13) -- (D24)
            node[edge, very near start] {$35$}
            node[edge, very near end] {$5$};
        \draw (D31) -- (D22)
            node[edge, very near start] {$19$};
        \draw (D31) -- (D42)
            node[edge, very near start, swap] {$35$}
            node[edge, very near end, swap] {$5$};
        \draw (D22) -- (D33);
        \draw (D24) -- (D33)
            node[edge, very near start] {$40$};
        \draw (D24) -- (D35)
            node[edge, very near start] {$10$}
            node[edge, very near end] {$20$};
        \draw (D42) -- (D33)
            node[edge, very near start, swap] {$40$};
        \draw (D42) -- (D53)
            node[edge, very near start, swap] {$10$}
            node[edge, very near end, swap] {$20$};
        \draw (D33) -- (D44);
        \draw (D35) -- (D44)
            node[edge, very near start] {$35$};
        \draw (D53) -- (D44)
            node[edge, very near start, swap] {$35$};
        \draw (D44) -- (D55)
            node[edge, very near end, swap] {$55$};
[i+1\ i-1]
	\end{tikzpicture}
	\endpgfgraphicnamed
	\caption{
		 The partition of vertices of $\Sigma$
         by distance from a pair of vertices $u, v$ at distance $2$.
		 The part that is at distance $i$ from $u$ and distance $j$ from $v$
		 has size $p_{ij}^2$.
		 As the graph is bipartite,
         the intersection number $p_{ij}^2$ is nonzero
         only when $i+j$ is even.
         Moreover, there are no edges within each part.
		 It is natural to consider $[1\ 1\ 1]$
         for $w$ at distance $2$ from both $u$ and $v$,
		 see Lemma~\ref{lem:5554:222}.
	 }
	\label{fig:dp22}
\end{figure}
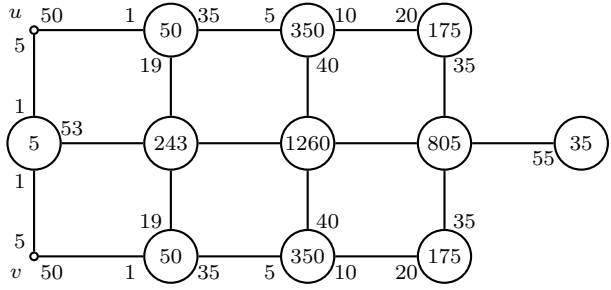

\begin{lemma} \label{lem:5554:222}
Let $\Sigma$ be a distance-regular graph
with intersection array \eqref{eqn:5554:ia},
and $u, v, w$ be vertices of $\Sigma$ that are pairwise at distance $2$.
Then $\sV{u & v & w}{1 & 1 & 1} \le 1$.
\end{lemma}

\begin{proof}
We consider the triple intersection numbers $[i\ j\ h]$
that correspond to $(u,v,w)$.
Since the graph $\Sigma$ is bipartite,
we have $[i\ j\ h] = 0$ whenever any of the sums $i+j$, $j+h$, $h+i$ is odd.
As $q^3_{11} = q^4_{11} = 0$,
Theorem~\ref{thm:krein0} gives us two additional equations
to the system \eqref{eqn:triple},
thus allowing us to express the solution of the system
in terms of a single parameter $\alpha = [1\ 1\ 1]$.
In particular, we obtain
$$
[5\ 5\ 5] = 20 - 12\alpha .
$$
The integrality and nonnegativity of $[5\ 5\ 5]$
now gives $\alpha \le \lfloor 5/3 \rfloor = 1$.
\end{proof}

\noindent
{\bf Note.} It can also be shown
with a method similar to the one used in Lemma~\ref{lem:5554:222}
that the graph $[\Sigma_5(u)]_2$ for a vertex $u \in V\Sigma$
(i.e., the graph of vertices at distance $5$ from a vertex $u$,
with adjacency corresponding to distance $2$ in $\Sigma$)
is strongly regular with parameters
$(v, k, \lambda, \mu) = (210, 99, 48, 45)$.
A strongly regular graph with such parameters
has been constructed by M.~Klin~\cite{kprwz10}.

\medskip

\begin{theorem} \label{thm:5554:nonex}
A distance-regular graph $\Sigma$ with intersection array \eqref{eqn:5554:ia}
does not exist.
\end{theorem}

\begin{proof}
Let $u$ and $v$ be vertices of $\Sigma$ at distance $2$,
see Figure~\ref{fig:dp22},
and let $\{i\ j\}$ denote the set of vertices
at distances $i$ and $j$ from $u$ and $v$, respectively.
There are $p_{11}^2 (k-2) = 5 \cdot 53 = 265$ edges
between the sets $\{1\ 1\}$ and $\{2\ 2\}$.
However, the cardinality of the latter set is $p^2_{22} = 243 < 265$,
so there is a vertex $w \in \{2\ 2\}$
that has at least two neighbours in $\{1\ 1\}$,
i.e., $\sV{u & v & w}{1 & 1 & 1} \ge 2$,
which is in contradiction with Lemma~\ref{lem:5554:222}.
Hence, the graph $\Sigma$ does not exist.
\end{proof}

\section*{Acknowledgements}
I would like to thank Michael Lang
for bringing the intersection array \eqref{eqn:5554:ia} to my attention
and for noticing some bugs and proposing new functionality
for the {\tt sage-drg} package.

\begin{appendices}
\section{Description of the {\tt sage-drg} package} \label{app:pack}
\appsection{Installation}
\label{sapp:install}

The {\tt sage-drg} package~\cite{v18} can be installed by
cloning the {\tt git} repository or extracting the ZIP file
and making sure that Sage sees the {\tt drg} directory
(e.g., by starting it from the package's root directory,
or by copying or linking the {\tt drg} directory into the package library
used by the copy of Python in Sage's installation directory).
Once Sage is run, the package can be imported.

\smallskip
\begin{python}
sage: import drg
\end{python}

The central class is \pythoninline{drg.DRGParameters},
which can be given an intersection array
in the form of two lists or tuples of the same length.

\smallskip
\begin{python}
sage: syl = drg.DRGParameters([5, 4, 2], [1, 1, 4])
sage: syl
Parameters of a distance-regular graph with intersection array
{5, 4, 2; 1, 1, 4}
\end{python}
Instead of an intersection array,
parameters $(k, \lambda, \mu)$ for a strongly regular graph
or classical parameters $(d, b, \alpha, \beta)$ (see~\cite[\S6]{bcn89})
may also be specified.

\smallskip
\begin{python}
sage: petersen = drg.DRGParameters(3, 0, 1)
sage: petersen
Parameters of a distance-regular graph with intersection array {3, 2; 1, 1}
sage: q7 = drg.DRGParameters(7, 1, 0, 1)
sage: q7
Parameters of a distance-regular graph with intersection array
{7, 6, 5, 4, 3, 2, 1; 1, 2, 3, 4, 5, 6, 7}
\end{python}
The intersection array (given in any of the forms above)
may also contain variables.
Substitution of variables is possible using the \pythoninline{subs} method.
Note that the diameter must be constant.

\smallskip
\begin{python}
sage: r = var("r")
sage: fam = drg.DRGParameters([2*r^2*(2*r+1), (2*r-1)*(2*r^2+r+1), 2*r^2],
                              [1, 2*r^2 , r*(4*r^2-1)])
sage: fam1 = fam.subs(r == 1)
sage: fam1
Parameters of a distance-regular graph with intersection array
{6, 4, 2; 1, 2, 3}
\end{python}

\appsection{Parameter computation}
\label{sapp:computation}

As a \pythoninline{drg.DRGParameters} object is being constructed,
its intersection numbers are computed.
If any of them is determined to be negative or nonintegral,\footnote{
Non-constant expressions are also checked for nonnegativity
-- if there is no assignment of real numbers to the variables
such that the value of the expression is nonnegative,
then the expression is marked as negative
and cannot appear as, e.g., an intersection number.}
then an exception is thrown and the construction is aborted.
Several other conditions are also checked,
for example that the sequences $\{b_i\}_{i=0}^d$ and $\{c_i\}_{i=0}^d$
are non-ascending and non-descending, respectively,
and that $b_j \ge c_i$ if $i+j \le d$~\cite[Prop.~4.1.6(ii)]{bcn89}.
The handshake lemma is also checked for each subconstituent~%
\cite[Lem.~4.3.1]{bcn89}.
If the graph is determined to be antipodal,
then the covering index is also checked for integrality.

The number of vertices, their valency and the diameter of the graph
can be obtained with the \pythoninline{order}, \pythoninline{valency},
and \pythoninline{diameter} methods.

\smallskip
\begin{python}
sage: syl.order()
36
sage: syl.valency()
5
sage: syl.diameter()
3
\end{python}

The entire array of intersection numbers can be obtained
with the \pythoninline{pTable} method,
which returns a \pythoninline{drg.Array3D} object
implementing a three-dimensional array.

\smallskip
{\setstretch{0.9}
\begin{python}
sage: syl.pTable()
0: [ 1  0  0  0]
   [ 0  5  0  0]
   [ 0  0 20  0]
   [ 0  0  0 10]

1: [0 1 0 0]
   [1 0 4 0]
   [0 4 8 8]
   [0 0 8 2]

2: [ 0  0  1  0]
   [ 0  1  2  2]
   [ 1  2 11  6]
   [ 0  2  6  2]

3: [ 0  0  0  1]
   [ 0  0  4  1]
   [ 0  4 12  4]
   [ 1  1  4  4]
\end{python}
}
The subsets of intersection numbers
$\{a_i\}_{i=1}^d$, $\{b_i\}_{i=0}^{d-1}$, $\{c_i\}_{i=1}^d$,
and $\{k_i\}_{i=0}^d$
(where $k_i = p^0_{ii}$ is the number of vertices
at distance $i$ from any vertex of the graph)
can be obtained as tuples with the \pythoninline{aTable},
\pythoninline{bTable}, \pythoninline{cTable},
and \pythoninline{kTable} methods, respectively.
There is also a method \pythoninline{intersectionArray}
returning the entire intersection array as a pair of tuples.

\smallskip
\begin{python}
sage: syl.aTable()
(0, 2, 1)
sage: syl.bTable()
(5, 4, 2)
sage: syl.cTable()
(1, 1, 4)
sage: syl.kTable()
(1, 5, 20, 10)
sage: syl.intersectionArray()
((5, 4, 2), (1, 1, 4))
\end{python}

Eigenvalues can be computed using the \pythoninline{eigenvalues} method.

\smallskip
\begin{python}
sage: syl.eigenvalues()
(5, 2, -1, -3)
\end{python}
Eigenvalues are sorted in the decreasing order;
if there is a variable in the intersection array,
then the order is derived under the assumption that the variable takes
a (large enough) positive value.
If there is more than one variable,
a warning is issued that the given ordering is not necessarily correct.
This can be avoided by explicitly specifying an order of the variables
using the \pythoninline{set_vars} method
(see the 
\href{https://github.com/jaanos/sage-drg/blob/master/jupyter/DRG-d3-2param.ipynb}{\tt jupyter/DRG-d3-2param.ipynb}
notebook for an example).
The ordering of eigenvalues
(and thus the ordering of corresponding parameters)
can also be changed later using the \pythoninline{reorderEigenvalues} method,
which accepts an ordering of the indices of the nontrivial eigenvalues
(i.e., integers from $1$ to $d$).

Once the ordering of eigenvalues is determined,
the cosine sequences and multiplicities of the eigenvalues can be computed
using the \pythoninline{cosineSequences}
and \pythoninline{multiplicities} methods.
The multiplicities are checked to be integral.

\smallskip
\begin{python}
sage: syl.cosineSequences()
[    1     1     1     1]
[    1   2/5 -1/20  -1/5]
[    1  -1/5  -1/5   2/5]
[    1  -3/5   1/5  -1/5]
sage: syl.multiplicities()
(1, 16, 10, 9)
\end{python}

The eigenmatrix and dual eigenmatrix can be computed using
the \pythoninline{eigenmatrix} and \pythoninline{dualEigenmatrix} methods.
The \pythoninline{is_formallySelfDual} method checks whether
the two matrices are equal (using the current ordering of eigenvalues).

\smallskip
\begin{python}
sage: syl.eigenmatrix()
[ 1  5 20 10]
[ 1  2 -1 -2]
[ 1 -1 -4  4]
[ 1 -3  4 -2]
sage: syl.dualEigenmatrix()
[    1    16    10     9]
[    1  32/5    -2 -27/5]
[    1  -4/5    -2   9/5]
[    1 -16/5     4  -9/5]
sage: syl.is_formallySelfDual()
False
\end{python}

The Krein parameters can be computed
using the \pythoninline{kreinParameters} method,
which returns a \pythoninline{drg.Array3D} object.
The Krein parameters are checked to be nonnegative.

\smallskip
{\setstretch{0.9}
\begin{python}
sage: syl.kreinParameters()
0: [ 1  0  0  0]
   [ 0 16  0  0]
   [ 0  0 10  0]
   [ 0  0  0  9]

1: [   0    1    0    0]
   [   1 44/5 22/5  9/5]
   [   0 22/5    2 18/5]
   [   0  9/5 18/5 18/5]

2: [     0      0      1      0]
   [     0 176/25   16/5 144/25]
   [     1   16/5      4    9/5]
   [     0 144/25    9/5  36/25]

3: [   0    0    0    1]
   [   0 16/5 32/5 32/5]
   [   0 32/5    2  8/5]
   [   1 32/5  8/5    0]
\end{python}
}

Classical parameters can be computed
using the \pythoninline{is_classical} method,
which returns a list of all tuples of classical parameters,
or \pythoninline{False} if the graph is not classical.

\smallskip
\begin{python}
sage: fam1.is_classical()
[(3, 1, 0, 2)]
\end{python}

The method \pythoninline{genPoly_parameters} returns the tuple $(g, s, t)$
if the parameters correspond to those of a collinearity graph
of a generalized $g$-gon of order $(s, t)$,
or \pythoninline{(False, None, None)}
if there is no such generalized $g$-gon.
See~\cite{t59} and~\cite[\S6.5]{bcn89}
for definitions of generalized polygons and some results.
\\[3mm]
\begin{minipage}{\linewidth}
\begin{python}
sage: drg.DRGParameters([6, 4, 4], [1, 1, 3]).genPoly_parameters()
(6, 2, 2)
\end{python}
\end{minipage}
Note that the existence of a strongly regular graph
for which $(g, s, t)$ are defined
does not imply the existence of a corresponding generalized quadrangle.
A distance-regular graph $\G$ of diameter at least $3$
has these parameters defined precisely when
$\G$ is isomorphic to the collinearity graph
of a corresponding generalized $g$-gon.

All the methods mentioned above store their results,
so subsequent calls will not redo the computations.
Note that one does not need to call the methods in the order given here
-- if some required computation has not been done before,
it will be performed when needed.
Where applicable, the methods above also take three named boolean parameters
\pythoninline{expand}, \pythoninline{factor}, and \pythoninline{simplify}
(all set to \pythoninline{False} by default),
which control how the returned expression(s) will be manipulated.
In the case when there are no variables in use,
setting these parameters has no effect.

\appsection{Parameters of derived graphs}
\label{sapp:derived}

In some cases,
the parameters of a distance-regular graph
imply the existence of another distance-regular graph
which can be derived from the original graph.
This is true for imprimitive graphs (i.e., antipodal or bipartite),
but sometimes, new distance-regular graphs
can be obtained by taking subgraphs or by merging classes.

The antipodality of a graph can be checked
with the \pythoninline{is_antipodal} method,
which returns the covering index for antipodal graphs,
and \pythoninline{False} otherwise.
The parameters of the antipodal quotient of an antipodal graph
can then be obtained with the \pythoninline{antipodalQuotient} method.

\smallskip
\begin{python}
sage: q7.is_antipodal()
2
sage: q7.antipodalQuotient()
Parameters of a distance-regular graph with intersection array
{7, 6, 5; 1, 2, 3}
\end{python}

The bipartiteness of a graph can be checked
with the \pythoninline{is_bipartite} method.
The parameters of the bipartite half of a bipartite graph
can then be obtained with the \pythoninline{bipartiteHalf} method.

\smallskip
\begin{python}
sage: q7.is_bipartite()
True
sage: q7.bipartiteHalf()
Parameters of a distance-regular graph with intersection array
{21, 10, 3; 1, 6, 15}
\end{python}

In some cases, distance-regularity of the local graph can be established
(for instance, for tight graphs, see~\cite{jkt00}).
In these cases, the parameters of the local graph
can then be obtained with the \pythoninline{localGraph} method.
\\[3mm]
\begin{minipage}{\linewidth}
\begin{python}
sage: drg.DRGParameters([27, 10, 1], [1, 10, 27]).localGraph()
Parameters of a distance-regular graph with intersection array
{16, 5; 1, 8}
\end{python}
\end{minipage}

Similarly, the distance-regularity of a subconstituent
(i.e., a graph induced by vertices at a given distance from a vertex)
can be established in certain cases.
Their parameters can be obtained
using the \pythoninline{subconstituent} method.
Usually, distance-regularity is derived from triple intersection numbers
(see \appSection~\ref{sapp:triple}),
which are not computed by default.
To force this computation, the parameter \pythoninline{compute}
can be set to \pythoninline{True}.

\smallskip
\begin{python}
sage: drg.DRGParameters([204, 175, 48, 1],
                        [1, 12, 175, 204]).subconstituent(2, compute = True)
Parameters of a distance-regular graph with intersection array
{144, 125, 32, 1; 1, 8, 125, 144}
\end{python}
Note that calling \pythoninline{localGraph()}
is equivalent to calling \pythoninline{subconstituent(1)}.
The \pythoninline{localGraph} method
also accepts the \pythoninline{compute} parameter.

The complement of a strongly regular graph is also strongly regular.
If the complement is connected,
its parameters can be obtained
with the \pythoninline{complementaryGraph} method.

\smallskip
\begin{python}
sage: petersen.complementaryGraph()
Parameters of a distance-regular graph with intersection array {6, 2; 1, 4}
\end{python}

Sometimes,
merging classes of the underlying association scheme
yields a new distance-regular graph.
Its parameters
(or the parameters of a connected component
if the resulting graph is disconnected)
can be obtained with the \pythoninline{mergeClasses} method,
which takes the indices of classes
which will be merged into the first class of the new scheme
(i.e., the distances in the original graph
which will correspond to adjacency in the new graph).

\smallskip
\begin{python}
sage: q7.mergeClasses(2, 3, 6)
Parameters of a distance-regular graph with intersection array
{63, 30, 1; 1, 30, 63}
\end{python}
Note that \pythoninline{mergeClasses(2)}
gives the parameters of the bipartite half for bipartite graphs,
and of the complement for non-antipodal strongly regular graphs.

A dictionary mapping the merged indices to parameters of a new graphs
for all possibilities
can be obtained using the \pythoninline{distanceGraphs} method.

\smallskip
{\setstretch{0.9}
\begin{python}
sage: q7.distanceGraphs()
{(1, 2): Parameters of a distance-regular graph with intersection array
         {28, 15, 6, 1; 1, 6, 15, 28},
 (1, 2, 3, 4, 5, 6): Parameters of a distance-regular graph
                     with intersection array {126, 1; 1, 126},
 (1, 3, 5): Parameters of a distance-regular graph with intersection array
            {63, 62, 1; 1, 62, 63},
 (1, 3, 5, 7): Parameters of a distance-regular graph
               with intersection array {64, 63; 1, 64},
 (1, 4, 5): Parameters of a distance-regular graph with intersection array
            {63, 32, 1; 1, 32, 63},
 (1, 5): Parameters of a distance-regular graph with intersection array
        {28, 27, 16; 1, 12, 28},
 (1, 7): Parameters of a distance-regular graph with intersection array
         {8, 7, 6, 5; 1, 2, 3, 8},
 (2,): Parameters of a distance-regular graph with intersection array
       {21, 10, 3; 1, 6, 15},
 (2, 3, 6): Parameters of a distance-regular graph with intersection array
            {63, 30, 1; 1, 30, 63},
 (2, 4, 6): Parameters of a distance-regular graph with intersection array
            {63; 1},
 (2, 6): Parameters of a distance-regular graph with intersection array
         {28, 15; 1, 12},
 (3, 7): Parameters of a distance-regular graph with intersection array
         {36, 35, 16; 1, 20, 36},
 (4,): Parameters of a distance-regular graph with intersection array
       {35, 16; 1, 20},
 (6,): Parameters of a distance-regular graph with intersection array
       {7, 6, 5; 1, 2, 3},
 (7,): Parameters of a distance-regular graph with intersection array
       {1; 1}}
\end{python}
}

\appsection{Feasibility checking}
\label{sapp:feasibility}

To check whether a given parameter set is feasible,
the \pythoninline{check_feasible} method may be called.
This method calls other \pythoninline{check_*} methods
which perform the actual checks.
Selected checks may also be skipped
by providing a parameter \pythoninline{skip}
with a list of strings identifying checks to be skipped.

\underline{\pythoninline{sporadic}}.
The \pythoninline{check_sporadic} method checks
whether the intersection array matches one from a list of intersection arrays
for which nonexistence of a corresponding graph has been proven,
but does not belong to any infinite family.
If so, the parameter set is reported as infeasible.
Currently, the list includes:
\begin{itemize}
\itemsep 0pt
\item $\{14, 12; 1, 4\}$, cf.~Wilbrink and Brouwer~\cite{wb83},
\item $\{16, 12; 1, 6\}$, cf.~Bussemaker et al.~\cite{bhmw89},
\item $\{21, 18; 1, 7\}$, cf.~Haemers~\cite{h93},
\item $\{30, 21; 1, 14\}$, cf.~Bondarenko, Prymak and Radchenko~\cite{bpr17},
\item $\{32, 21; 1, 16\}$, cf.~Azarija and Marc~\cite{am18},
\item $\{38, 27; 1, 18\}$, cf.~Degraer~\cite{d07},
\item $\{40, 27; 1, 20\}$, cf.~Azarija and Marc~\cite{am16},
\item $\{57, 56; 1, 12\}$, cf.~Gavrilyuk and Makhnev~\cite{gm05},
\item $\{67, 56; 1, 2\}$, cf.~Brouwer and Neumaier~\cite{bn81},
\item $\{116, 115; 1, 20\}$, cf.~Makhnev~\cite{m17},
\item $\{153, 120; 1, 60\}$, cf.~Bondarenko et al.~\cite{bmprv18},
\item $\{165, 128; 1, 66\}$, cf.~Makhnev~\cite{m02},
\item $\{486, 320; 1, 243\}$, cf.~Makhnev~\cite{m02},
\item $\{5, 4, 3; 1, 1, 2\}$, cf.~Fon-Der-Flaass~\cite{fdf93b},
\item $\{11, 10, 10; 1, 1, 11\}$ (projective plane of order $10$),
    cf.~Lam, Thiel and Swiercz~\cite{lts89},
\item $\{13, 10, 7; 1, 2, 7\}$, cf.~Coolsaet~\cite{c95},
\item $\{18, 12, 1; 1, 2, 18\}$
    (generalized quadrangle of order $(6, 3)$ minus a spread),
    cf.~\cite[Prop.~12.5.2]{bcn89} and~\cite[6.2.2]{pt09},
\item $\{20, 10, 10; 1, 1, 2\}$ (projective plane of order $10$),
    cf.~Lam, Thiel and Swiercz~\cite{lts89},
\item $\{21, 16, 8; 1, 4, 14\}$, cf.~Coolsaet~\cite{c05},
\item $\{22, 16, 5; 1, 2, 20\}$, cf.~Sumalroj and Worawannotai~\cite{sw16},
\item $\{27, 20, 10; 1, 2, 18\}$,
    cf.~Brouwer, Sumalroj and Worawannotai~\cite{bsw16},
\item $\{36, 28, 4; 1, 2, 24\}$,
    cf.~Brouwer, Sumalroj and Worawannotai~\cite{bsw16},
\item $\{39, 24, 1; 1, 4, 39\}$, cf.~Bang, Gavrilyuk and Koolen~\cite{bgk18},
\item $\{45, 30, 7; 1, 2, 27\}$, cf.~Gavrilyuk and Makhnev~\cite{gm13},
\item $\{52, 35, 16; 1, 4, 28\}$, cf.~Gavrilyuk and Makhnev~\cite{gm12},
\item $\{55, 36, 11; 1, 4, 45\}$, cf.~Gavrilyuk~\cite{g11},
\item $\{56, 36, 9; 1, 3, 48\}$, cf.~Gavrilyuk~\cite{g11},
\item $\{69, 48, 24; 1, 4, 46\}$, cf.~Gavrilyuk and Makhnev~\cite{gm12},
\item $\{74, 54, 15; 1, 9, 60\}$, cf.~Coolsaet and Jurišić~\cite{cj08},
\item $\{105, 102, 99; 1, 2, 35\}$, cf.~De Bruyn and Vanhove~\cite{dbv15},
\item $\{130, 96, 18; 1, 12, 117\}$, cf.~Jurišić and Vidali~\cite{jv17},
\item $\{135, 128, 16; 1, 16, 120\}$, see \appref{thm:135:nonex}{Theorem}{4},
\item $\{234, 165, 12; 1, 30, 198\}$, see \appref{thm:234:nonex}{Theorem}{5},
\item $\{4818, 4248, 192; 1, 72, 4672\}$, cf.~Jurišić and Vidali~\cite{jv17},
\item $\{5928, 5920, 5888; 1, 5, 741\}$,
    cf.~De Bruyn and Vanhove~\cite{dbv15},
\item $\{120939612, 120939520, 120933632; 1, 65, 1314561\}$,
    cf.~De Bruyn and Vanhove~\cite{dbv15},
\item $\{97571175, 97571080, 97569275; 1, 20, 1027065\}$,
    cf.~De Bruyn and Vanhove~\cite{dbv15},
\item $\{290116365, 290116260, 290100825; 1, 148, 2763013\}$,
    cf.~De Bruyn and Vanhove~\cite{dbv15},
\item $\{5, 4, 3, 3; 1, 1, 1, 2\}$, cf.~Fon-Der-Flaass~\cite{fdf93a},
\item $\{10, 9, 1, 1; 1, 1, 9, 10\}$, cf.~\cite[Prop.~11.4.5]{bcn89},
\item $\{32, 27, 6, 1; 1, 6, 27, 32\}$, cf.~Soicher~\cite{s17},
\item $\{32, 27, 9, 1; 1, 3, 27, 32\}$, cf.~Soicher~\cite{s17},
\item $\{56, 45, 20, 1; 1, 4, 45, 56\}$, cf.~\cite[Prop.~11.4.5]{bcn89},
\item $\{55, 54, 50, 35, 10; 1, 5, 20, 45, 55\}$,
    see \appref{thm:5554:nonex}{Theorem}{7}, and
\item $\{15, 14, 12, 6, 1, 1; 1, 1, 3, 12, 14, 15\}$,
    cf.~Ivanov and Shpectorov~\cite{is90}.
\end{itemize}

\underline{\pythoninline{family}}.
The \pythoninline{check_family} method checks
whether the intersection array matches one from a list
of infinite families of intersection arrays
for which nonexistence of corresponding graphs has been proven.
If so, the parameter set is reported as infeasible.
Currently, the list includes:
\begin{itemize}
\itemsep 0pt
\item $\{r^2 (r+3), (r+1)(r^2+2r-2); 1, r(r+1)\}$ with $r \ge 3$, $r \ne 4$,
    cf.~Bondarenko and Radchenko~\cite{br13},
\item $\{(2r^2 - 1)(2r+1), 4r(r^2-1), 2r^2; 1, 2(r^2-1), r(4r^2-2)\}$
    with $r \ge 2$, cf.~Jurišić and Vidali~\cite{jv12},
\item $\{2r^2 (2r+1), (2r-1)(2r^2+r+1), 2r^2; 1, 2r^2, r(4r^2-1)\}$
    with $r \ge 2$, cf.~Jurišić and Vidali~\cite{jv12},
\item $\{4r^3 + 8r^2 + 6r + 1, 2r(r+1)(2r+1), 2r^2 + 2r + 1;
         1, 2r(r+1), (2r+1)(2r^2+2r+1)\}$ with $r \ge 1$,
    cf.~Coolsaet and Jurišić~\cite{cj08},
\item $\{(2r\!+\!1)(4r\!+\!1)(4t\!-\!1),
         8r(4rt\!-\!r\!+\!2t), (r\!+\!t)(4r\!+\!1);
         1, (r\!+\!t)(4r\!+\!1), 4r(2r\!+\!1)(4t\!-\!1)\}$ with $r, t \ge 1$,
    see \appref{thm:2param:nonex}{Theorem}{3},
\item $\{(r+1)(r^3-1), r(r-1)(r^2+r-1), r^2-1;
         1, r(r+1), (r^2-1)(r^2+r-1)\}$ with $r \ge 3$,
    cf.~Urlep~\cite{u12},
\item $\{r^2 (rt+t+1), (r^2-1)(rt+1), r(r-1)(t+1), 1;
         1, r(t+1), (r^2-1)(rt+1), r^2 (rt+t+1)\}$
    with $r \ge 3$ and $(r, t) \ne (3, 1), (3, 3), (4, 2)$,
    cf.~Jurišić and Koolen~\cite{jk11}, and
\item $\{2r^2+r, 2r^2+r-1, r^2, r, 1; 1, r, r^2, 2r^2+r-1, 2r^2+r\}$
    with $r \ge 2$, cf.~Coolsaet, Jurišić and Koolen~\cite{cjk08}.
\end{itemize}

\underline{\pythoninline{2graph}}.
The \pythoninline{check_2graph} method
checks conditions related to two-graphs,
cf.~\cite[Thm.~1.5.6]{bcn89}.
For strongly regular graphs with parameters $(v, k, \lambda, \mu)$,
for which $v = 2(2k - \lambda - \mu)$ holds,
it records the parameters $(v-1, 2(k-\mu), k+\lambda-2\mu, k-\mu)$
of a strongly regular graph to be checked for feasibility later.
For Taylor graphs (i.e., antipodal double covers of diameter $3$)
for which $a_1 > 0$ holds,
it checks whether $a_1$ is even
and whether the number of vertices $n$ is a multiple of $4$,
and then records the parameters $(k, a_1, (3a_1-k-1)/2, a_1/2)$
of a strongly regular graph as those af the local graph
to be checked for feasibility later, cf.~\cite[Thm.~1.5.3]{bcn89}.

\underline{\pythoninline{classical}}.
The \pythoninline{check_classical} method checks
whether any of the classical parameters for the parameter set
match some from a list of infinite families of classical paramters
for which nonexistence of corresponding graphs has been proven.
If so, the parameter set is reported as infeasible.
Currently, the list only includes two sets of classical parameters:
\begin{itemize}
\itemsep 0pt
\item $(d, b, \alpha, \beta) = (d, -2, -2, ((-2)^{d+1} - 1)/3)$
    with $d \ge 4$, cf. Huang, Pan and Weng~\cite{hpw15}, and
\item $(d, b, \alpha, \beta) =
    (d, -r, -r/(r-1), r + r^2 ((-r)^{d-1} - 1) / (r^2-1))$
    with $d \ge 4$, $r \ge 2$,
    cf.~De Bruyn and Vanhove~\cite{dbv15}.
\end{itemize}
Additionally, the method checks whether nonexistence
can be derived from one of the following characterizations:
\begin{itemize}
\itemsep 0pt
\item a characterization of Grassmann graphs by Metsch~\cite[Thm.~2.3]{m95},
\item a characterization of bilinear forms graphs
    by Metsch~\cite[Prop.~2.2]{m99},
\item a characterization of graphs
    with classical parameters $(d, b, \alpha, \beta)$ and $d \ge 4$, $b < 0$
    by Weng~\cite[Thm.~10.3]{w99}, and
\item a characterization of graphs
    with classical parameters $(d, b, \alpha, \beta)$
    and $d \ge 3$, $a_1 = 0$, $a_2 > 0$
    by Pan and Weng~\cite[Thm.~2.1]{pw09}.
\end{itemize}

\underline{\pythoninline{combinatorial}}.
The \pythoninline{check_combinatorial} method
checks various combinatorial conditions:
\begin{itemize}
\itemsep 0pt
\item a graph with $b_1 = 1$ must be a cycle or a cocktail party graph,
\item Godsil's diameter bound~\cite[Lem.~5.3.1]{bcn89},
\item a lower bound for $c_3$~\cite[Thm.~5.4.1]{bcn89},
\item a condition for $b_1 = b_i$~\cite[Prop.~5.4.4]{bcn89},
\item a lower bound for $a_1$~\cite[Prop.~5.5.1]{bcn89},
\item a handshake lemma for Pappus subgraphs, cf.~Koolen~\cite{k92}
\item Turán's theorem~\cite[Lem.~5.6.4]{bcn89},
\item a counting argument by Lambeck~\cite{l93},
\item two lower bounds for the size of the last subconstituent~%
    \cite[Props.~5.6.1,~5.6.3]{bcn89},
\item handshake lemmas for the numbers of edges and triangles~%
    \cite[Lem.~4.3.1]{bcn89}, and
\item a condition for $p^2_{dd} = 0$~\cite[Prop.~5.7.1]{bcn89}.
\end{itemize}
Additionally,
the method checks whether a condition
for the existence of cliques of size $a_1 + 2$
is satisified~\cite[Prop.~4.3.2]{bcn89}
and performs additional checks:
\begin{itemize}
\itemsep 0pt
\item a divisibility check
    if the last subconstituent is a union of cliques~%
    \cite[Prop.~4.3.2(ii)]{bcn89},
\item a handshake lemma for maximal cliques~\cite[Prop.~4.3.3]{bcn89}, and
\item two inequalities from counting arguments~\cite[Prop.~4.3.3]{bcn89}.
\end{itemize}

\underline{\pythoninline{conference}}.
For a strongly regular graph with parameters $(v, k, \lambda, \mu)$,
where $k = 2\mu$ and $\mu = k-\lambda-1$,
the \pythoninline{check_conference} method
checks whether $n \not\equiv 1 \pmod{4}$ and $n$ is a sum of two squares,
cf.~\cite[\S1.3]{bcn89}.

\underline{\pythoninline{geodeticEmbedding}}.
For a distance-regular graph with intersection array
$\{2b, b, 1; 1, 1,$ $2b\}$,
the \pythoninline{check_geodeticEmbedding} method
checks whether $b \le 4$, cf.~\cite[Prop.~1.17.3]{bcn89}.

\underline{\pythoninline{2design}}.
For a distance-regular graph with intersection array
$\{r \mu + 1, (r-1) \mu, 1; 1, \mu,$ $r \mu + 1\}$,
the \pythoninline{check_2design} method
checks that a corresponding $2$-design exists,
cf.~\cite[Prop. 1.10.5]{bcn89}.

\underline{\pythoninline{hadamard}}.
For a distance-regular graph with intersection array
$\{2\mu, 2\mu - 1, \mu, 1; 1, \mu,$ $2\mu - 1, 2\mu\}$ with $\mu > 1$,
the \pythoninline{check_hadamard} method
checks whether $\mu$ is even,
i.e., whether a Hadamard matrix of order $2\mu$ can exist,
cf.~\cite[Cor.~1.8.2]{bcn89}.

\underline{\pythoninline{antipodal}}.
For an antipodal cover of even diameter at least $4$,
the \pythoninline{check_antipodal} method
checks whether its quotient satisfies necessary conditions
for the existence of a cover, cf.~\cite[Prop.~4.2.7]{bcn89}.

\underline{\pythoninline{genPoly}}.
The \pythoninline{check_genPoly} method
checks conditions related to generalized polygons.
First, it checks whether the conditions
for the existence of cliques of size $a_1 + 2$ have been checked,
and calls \pythoninline{check_combinatorial} otherwise
to obtain this information.
If the existence of cliques has been established
and the intersection array matches that of a collinearity graph
of a generalized $g$-gon with parameter $(s, t)$,
then we know that we should indeed have such a graph,
and the following conditions are checked:
\begin{itemize}
\itemsep 0pt
\item Feit-Higman theorem~\cite[Thm.~6.5.1]{bcn89},
\item $(s+t) | st(s+1)(t+1)$ for generalized quadrangles~\cite[1.2.2]{pt09},
    and
\item Bruck-Ryser theorem for thin generalized hexagons~%
    \cite[Thm.~1.10.4]{bcn89}.
\end{itemize}

An antipodal distance-regular $r$-cover of diameter $3$
with $k = (r-1)(c_2+1)$
such that the existence of cliques of size $a_1 + 2$ has been established
corresponds to the collinearity graph
of a generalized quadrangle of order $(s, t) = (r-1, c_2+1)$
with a spread removed.
Therefore, in this case,
\pythoninline{check_genPoly} checks for feasibility
of such a generalized quadrangle (first two conditions above),
and also checks whether it can contain a spread,
see~\cite[Prop.~12.5.2]{bcn89} and~\cite[1.8.3]{pt09}.

\underline{\pythoninline{clawBound}}.
For a strongly regular graph,
the \pythoninline{check_clawBound} method checks the claw bound,
i.e., whether the graph must be the point graph
of an infeasible partial geometry,
cf.~Brouwer and Van Lint~\cite{bvl84}.

\underline{\pythoninline{terwilliger}}.
The \pythoninline{check_terwilliger} method
checks conditions related to Terwilliger graphs and induced quadrangles,
cf.~\cite[\S1.16]{bcn89}.
First, it checks the coclique bound by Koolen and Park~\cite[Thm.~3]{kp10};
if it is met with equality,
it checks wheter the graph can be a Terwilliger graph~%
\cite[Cor.~1.16.6]{bcn89}.
If the parameters imply that the graph contains an induced quadrangle,
then Terwilliger's diameter bound is checked~\cite[Thm.~5.2.1]{bcn89}.

\underline{\pythoninline{secondEigenvalue}}.
For a distance-regular graph with an eigenvalue equal to $b_1-1$,
the \pythoninline{check_secondEigenvalue} method
checks whether the graph belongs to the classification
given in~\cite[Thm.~4.4.11]{bcn89}.

\underline{\pythoninline{localEigenvalue}}.
The \pythoninline{check_localEigenvalue} method
checks conditions related to the eigenvalues of the local graph.
For a distance-regular graph of diameter $3$ that is not the dodecahedron,
the following conditions are checked:
\begin{itemize}
\itemsep 0pt
\item general bounds for the second largest and smallest eigenvalue
    of the local graph~\cite[Thm.~4.4.3]{bcn89},
\item a bound on eigenvalues of the local graph of a non-bipartite graph
    by Jurišić and Koolen~\cite{jk00},
\item the fundamental bound by Jurišić, Koolen and Terwilliger~\cite{jkt00},
    and
\item bounds on multiplicities of eigenvalues,
    see~\cite[Thm.~4.4.4]{bcn89},~\cite{gh92} and~\cite{gk95}.
\end{itemize}

If equality is met in the bounds of the second or third point above,
then the local graph is determined to be strongly regular
and is thus stored as such to be checked for feasibility later.

\underline{\pythoninline{absoluteBound}}.
The \pythoninline{check_absoluteBound} method
checks the absolute bound on the multiplicities of eigenvalues~%
\cite[Thm.~2.3.3]{bcn89}.

\bigskip
After running all the checks described above,
the \pythoninline{check_feasbile} method
calls itself on all already derived graphs
(antipodal quotient, bipartite half, complement, $2$-graph derivation,
subconstituents where applicable),
and then also on each parameter set for distance-regular graphs
obtained by merging classes.
To avoid repetitions,
a list of checked intersection arrays is maintained.
This step can be skipped
by setting the \pythoninline{derived} parameter to \pythoninline{False}.

If the parameter set is feasible (i.e., it passes all checks),
then \pythoninline{check_feasible} returns without error.
Otherwise, a \pythoninline{drg.InfeasibleError} exception is thrown
indicating the reason for nonexistence and providing a reference.

\smallskip
\begin{python}
sage: drg.DRGParameters(266, 220, 210).check_feasible()
...
InfeasibleError: complement: nonexistence by GavrilyukMakhnev05
\end{python}
Details on the given references
are available in the \pythoninline{drg.references} submodule.

\smallskip
\begin{python}
sage: import drg.references
sage: drg.references.refs["GavrilyukMakhnev05"]
{'authors': [('Gavrilyuk', ('Alexander', 'L.')),
             ('Makhnev', ('Alexander', 'Alexeevich'))],
 'fjournal': 'Doklady Akademii Nauk',
 'journal': 'Dokl. Akad. Nauk',
 'number': 6,
 'pages': (727, 730),
 'title': 'Krein graphs without triangles',
 'type': 'article',
 'volume': 403,
 'year': 2005}
\end{python}
Details on the nonexistence may also be extracted from the exception.

\smallskip
\begin{python}
sage: try:
....:     drg.DRGParameters([65, 44, 11], [1, 4, 55]).check_feasible()
....: except drg.InfeasibleError as ex:
....:     print("Part: %s" % (ex.part, ))
....:     print("Reason: %s" % ex.reason)
....:     for r, thm in ex.refs:
....:         ref = drg.references.refs[r]
....:         print("Authors: %s" % ref["authors"])
....:         print("Title: %s" % ref["title"])
....:         print("Theorem: %s" % thm)
....:
Part: ()
Reason: coclique bound exceeded
Authors: [('Koolen', ('Jack', 'H.')), ('Park', ('Jongyook',))]
Title: Shilla distance-regular graphs
Theorem: Thm. 3.
\end{python}

\appsection{Partitions and triple intersection numbers}
\label{sapp:triple}

For a given parameter set,
the distance partition corresponding to a vertex
can be obtained with the \pythoninline{distancePartition} method,
which returns a graph representing the distance partition.

\smallskip
\begin{python}[frame=t]
sage: dp = syl.distancePartition()
sage: dp
Distance partition of {5, 4, 2; 1, 1, 4}
sage: dp.show()
\end{python}
\includegraphics[width=0.5\textwidth]{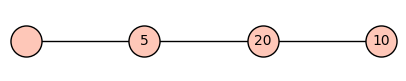}
\hrule\bigskip

A distance partition corresponding to two vertices can be obtained
by passing the distance between them
as an argument to \pythoninline{distancePartition}.

\smallskip
\begin{python}[frame=t]
sage: syl.distancePartition(1).show()
\end{python}
\includegraphics[width=0.5\textwidth]{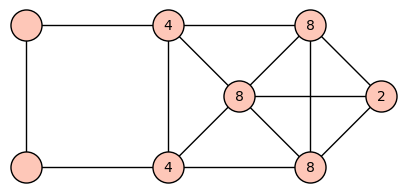}
\hrule\bigskip
Note that edges are shown between any pair of cells
such that their distances from either of the initial vertices
differ by at most $1$.
To show all the distance partitions corresponding to at most two vertices,
the \pythoninline{show_distancePartitions} method may be used (see below).

For a given triple of distances $(U, V, W)$ such that $p^W_{UV} > 0$,
the method \pythoninline{tripleEquations} gives the solution
to the system of equations \appeqref{eqn:triple}{1}
augmented by equations derived from \appref{thm:krein0}{Theorem}{1}
for each triple $(i, j, h)$ ($1 \le i, j, h \le d$) such that $q^h_{ij} = 0$.
The solution is returned as a \pythoninline{drg.Array3D} object.

\smallskip
{\setstretch{0.8}
\begin{python}
sage: syl.tripleEquations(1, 1, 2)
0: [0 0 0 0]
   [0 1 0 0]
   [0 0 0 0]
   [0 0 0 0]

1: [0 0 1 0]
   [0 0 0 0]
   [1 0 3 0]
   [0 0 0 0]

2: [0 0 0 0]
   [0 0 2 2]
   [0 2 2 4]
   [0 2 4 2]

3: [0 0 0 0]
   [0 0 0 0]
   [0 0 6 2]
   [0 0 2 0]
\end{python}
}

If the solution is not unique,
one or more parameters will be present in the solution.

\smallskip
{\setstretch{0.8}
\begin{python}
sage: syl.tripleEquations(1, 2, 3)
0: [0 0 0 0]
   [0 0 1 0]
   [0 0 0 0]
   [0 0 0 0]

1: [0 0 0 1]
   [0 0 0 0]
   [0 1 2 1]
   [0 0 0 0]

2: [        0         0         0         0]
   [        0         0         3         1]
   [        0        r2    r2 + 4 -2*r2 + 4]
   [        1   -r2 + 2   -r2 + 4  2*r2 + 1]

3: [        0         0         0         0]
   [        0         0         0         0]
   [        0   -r2 + 3   -r2 + 6  2*r2 - 1]
   [        0    r2 - 1        r2 -2*r2 + 3]
\end{python}
}

Parameters may also be set explicitly
by passing a \pythoninline{params} argument
with a dictionary mapping the name of the parameter
to the triple of distances it represents.

\smallskip
{\setstretch{0.8}
\begin{python}
sage: syl.tripleEquations(1, 3, 3, params = {"a": (3, 3, 3)})
0: [0 0 0 0]
   [0 0 0 1]
   [0 0 0 0]
   [0 0 0 0]

1: [0 0 0 1]
   [0 0 0 0]
   [0 0 4 0]
   [0 0 0 0]

2: [         0          0          0          0]
   [         0          0          4          0]
   [         0 -1/2*a + 4 -1/2*a + 4          a]
   [         0      1/2*a  1/2*a + 4     -a + 4]

3: [         0          0          0          0]
   [         0          0          0          0]
   [         0      1/2*a  1/2*a + 4     -a + 4]
   [         1 -1/2*a + 1     -1/2*a          a]
\end{python}
}

If the appropriate triple intersection number is computed to be zero,
an edge will not be shown in the distance partition.

\smallskip
\begin{python}[frame=t]
sage: syl.show_distancePartitions()
\end{python}
\begin{minipage}{0.49\textwidth}
\includegraphics[width=\textwidth]{syl-dp.png}
\end{minipage}
\begin{minipage}{0.49\textwidth}
\includegraphics[width=\textwidth]{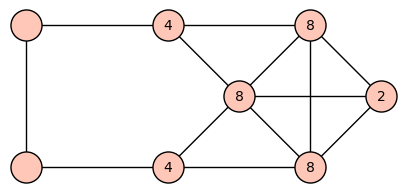}
\end{minipage} \\
\begin{minipage}{0.49\textwidth}
\includegraphics[width=\textwidth]{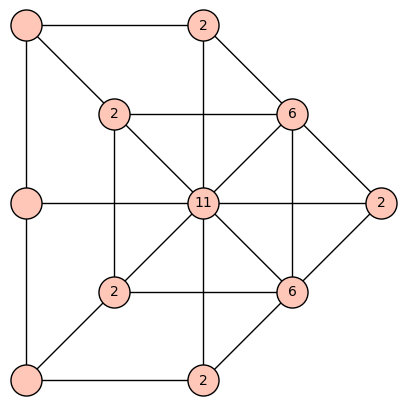}
\end{minipage}
\begin{minipage}{0.49\textwidth}
\centering
\includegraphics[width=0.56\textwidth]{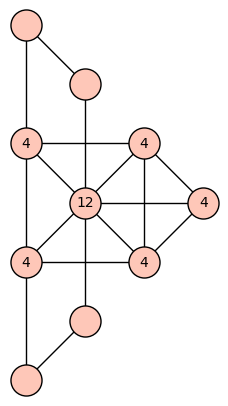}
\end{minipage}
\hrule

\end{appendices}

\end{document}